\documentclass[letterpaper, 10pt, conference]{ieeeconf}
\IEEEoverridecommandlockouts 
\usepackage{amsmath,amssymb,url}
\usepackage{graphicx,subfigure}
\usepackage{color}
\usepackage{cleveref}

\newcommand{\bracket}[1]{\ensuremath{\left[ #1 \right]}}

\newcommand{\parenth}[1]{\ensuremath{\left( #1 \right)}}

\newcommand{\tr}[1]{\mbox{tr}\ensuremath{\negthickspace\bracket{#1}}}
\newcommand{\deriv}[2]{\ensuremath{\frac{\partial #1}{\partial #2}}}
\newcommand{\G}{\ensuremath{\mathsf{G}}}
\newcommand{\SO}{\ensuremath{\mathsf{SO(3)}}}
\newcommand{\T}{\ensuremath{\mathsf{T}}}
\renewcommand{\L}{\ensuremath{\mathsf{L}}}
\newcommand{\so}{\ensuremath{\mathfrak{so}(3)}}

\renewcommand{\Re}{\ensuremath{\mathbb{R}}}

\newcommand{\D}{\ensuremath{\mathbf{D}}}
\newcommand{\pair}[1]{\ensuremath{\left\langle #1 \right\rangle}}
\newcommand{\met}[1]{\ensuremath{\langle\!\langle #1 \rangle\!\rangle}}
\newcommand{\Ad}{\ensuremath{\mathrm{Ad}}}
\newcommand{\ad}{\ensuremath{\mathrm{ad}}}
\newcommand{\g}{\ensuremath{\mathfrak{g}}}

\title{\LARGE \bf Variational Symplectic Accelerated Optimization on Lie Groups}

\author{%
    Taeyoung Lee, Molei Tao, and Melvin Leok%
    \thanks{Taeyoung Lee, Mechanical and Aerospace Engineering, George Washington University, Washington, DC 20052. {\tt tylee@gwu.edu}}%
    \thanks{Molei Tao, Mathematics, Georgia Institute of Technology, Atlanta, GA 30332. {\tt mtao@gatech.edu}}
    \thanks{Melvin Leok, Mathematics, University of California--San Diego, La Jolla, CA 92093. {\tt mleok@ucsd.edu}}
    \thanks{The research was partially supported by NSF under the grants CNS-1837382, CMMI-1760928, DMS-1813635, IIP-1747760, DMS-1847802, ECCS-1936776, and by AFOSR FA9550-18-1-0288 and DoD 13106725.}
}

\newcommand{\RomanNumeralCaps}[1]{\textup{\uppercase\expandafter{\romannumeral#1}}}

\newtheorem{prop}{Proposition}

\graphicspath{{./figs/}}

\begin{document}

\allowdisplaybreaks

\maketitle \thispagestyle{empty} \pagestyle{empty}

\begin{abstract}
There has been significant interest in generalizations of the Nesterov accelerated gradient descent algorithm due to its improved performance guarantee compared to the standard gradient descent algorithm, and its applicability to large scale optimization problems arising in deep learning.
A particularly fruitful approach is based on numerical discretizations of differential equations that describe the continuous time limit of the Nesterov algorithm, and a generalization involving time-dependent Bregman Lagrangian and Hamiltonian dynamics that converges at an arbitrarily fast rate to the minimum.
We develop a Lie group variational discretization based on an extended path space formulation of the Bregman Lagrangian on Lie groups, and analyze its computational properties with two examples in attitude determination and vision-based localization.  
\end{abstract}

\section{Introduction}

Nesterov's accelerated gradient descent algorithm~\cite{Nes83} was introduced in 1983, and it exhibits the convergence rate of $\mathcal{O}(1/k^2)$ when applied to a convex objective function, which is faster than the $\mathcal{O}(1/k)$ convergence rate of standard gradient descent methods.
It is shown in~\cite{nesterov2003introductory} that this rate of convergence is optimal for the class of first-order gradient methods.
This improved rate of convergence over the standard gradient method is referred to as acceleration, and there is a great interest in developing systematic approaches to the construction of efficient accelerated optimization algorithms, driven by potential applications in deep learning.

A continuous time limit of the Nesterov algorithm was studied in~\cite{su2016differential}, whose flow converges to the minimum at $\mathcal{O}(1/t^2)$, and this was generalized in~\cite{wibisono2016variational} using a time-dependent Bregman Lagrangian and Hamiltonian to obtain higher-order convergence of $\mathcal{O}(1/t^p)$ for arbitrary $p\geq 2$.
However, it has been shown that discretizing Bregman dynamics is not trivial as common discretizations fail to achieve the higher convergence rate guaranteed in the continuous time limit.
As such, there have been several attempts to construct accelerated optimization algorithms using geometric structure-preserving discretizations of the Bregman dynamics~\cite{betancourt2018symplectic}. 

A natural class\footnote{Note that other classes of discretization methods exist, such as those based on splitting (e.g., \cite{HaLuWa2006,tao2020variational}) and composition (e.g., \cite{tao2010nonintrusive}), and such approaches also arise in variational discretization \cite{MarWesAN01}.} 
of geometric numerical integrators~\cite{HaLuWa2006} for discretizing such Lagrangian or Hamiltonian systems is variational integrators~\cite{MarWesAN01,LeZh2009}.
They are constructed by a discrete analogue of Hamilton's variational principle, and therefore, their numerical flows are symplectic.
They also satisfy a discrete Noether's theorem that relates symmetries with momentum conservation properties, and further exhibit excellent exponentially long-time energy stability.
One complication is that such methods are typically developed for autonomous Lagrangian and Hamiltonian systems on the Euclidean space. 
To address this, variational integrators have been developed on a Lie group~\cite{LeeLeoCMAME07}, and time-adaptive Hamiltonian variational integrators have been proposed~\cite{DuScLe2021}.

In this paper, we focus on the optimization problem to minimize an objective function defined on an a Lie group. Optimization on a manifold or a Lie group appears in various areas of machine learning, engineering, and applied mathematics~\cite{hu2020brief,absil2009optimization}, and respecting the geometric structure of manifolds yields more accurate and efficient optimization schemes, when compared to methods based on embeddings in a higher-dimensional Euclidean space with algebraic constraints, or using local coordinates. 

In particular, we formulate a Bregman Lagrangian system on a Lie group, and we further discretize it using the extended Lie group variational integrator to construct an intrinsic accelerated optimization scheme, which inherits the desirable properties of variational integrators while also preserving the group structure. 
Compared with~\cite{DuScLe2021} where the evolution of the stepsize is prescribed, the proposed scheme adaptively adjusts the stepsize according to the extended variational principle at the cost of increased computational load. 
The resulting computational properties of the proposed approach are analyzed with two examples in attitude determination and vision-based localization,
where it is observed that the scheme exhibits an interesting convergence of the adaptive stepsize, and the variational discretization provides robustness against the choice of stepsize, which is exploited in the numerical experiments to improve computational efficiency. 
We also present benchmark studies against other discretization schemes applied to the Bregman dynamics, and other accelerated optimization schemes on a Lie group~\cite{tao2020variational}.

\section{Extended Lagrangian Mechanics}

This section presents Lagrangian mechanics for non-autonomous systems on a Lie group. 
It is referred to as \textit{extended} Lagrangian mechanics as the variational principle is extended to include reparamerization of time~\cite{MarWesAN01}.
These are developed in both of continuous-time and discrete-time formulations. The latter yields a \textit{Lie group variational integrator}~\cite{LeeLeoCMAME07}, which will be applied to accelerated optimization using the Bregman Lagrangian in the next section.

Consider  an $n$-dimensional Lie group $\G$.
Let $\g$ be the associated Lie algebra, or the tangent space at the identity, i.e., $\g = \T_e\G$.
Consider a left trivialization of the tangent bundle of the group $\T\G \simeq \G\times \g$, $(g,\dot g)\mapsto (g, L_{g^{-1}}\dot g)\equiv(g,\xi)$
More specifically, let $\L:\G\times\G\rightarrow\G$ be the left action defined such that $\L_g h = gh$ for $g,h\in\G$.
Then the left trivialization is a map $(g,\dot g)\mapsto (g, L_{g^{-1}}\dot g)\equiv(g,\xi)$, where $\xi\in\g$, and the kinematics equation can be written as
\begin{align}
    \dot g = g\xi. \label{eqn:g_dot}
\end{align}
Further, suppose $\g$ is equipped with an inner product $\pair{\cdot, \cdot}$, which induces an inner product on $\T_g\G$ via left trivialization.
For any $v,w\in\T_g\G$, $\pair{w,v}_{\T_g\G} = \pair{ \T_g \L_{g^{-1}} v, \T_g \L_{g^{-1}} w}_\g$. 
Given the inner product, we identify $\g\simeq \g^*$ and $\T_g \G \simeq \T^*_g \G\simeq G\times \g^*$ via the Riesz representation. Throughout this paper, the pairing is also denoted by the dot product $\cdot$.
Let $\mathbf{J}:\g\rightarrow\g^*$ be chosen such that $\pair{\mathbf{J}(\xi),\zeta}$ is positive-definite and symmetric as a bilinear form of $\xi,\zeta\in\g$.
Define the metric $\met{\cdot,\cdot}:\g\times\g\rightarrow\Re$ with $\met{\xi,\zeta} = \pair{\mathbf{J}(\xi),\zeta}$.
This serves as a left-invariant Riemmanian metric on $\G$.
Also $\|\xi\|^2 = \met{\xi,\xi}$ for any $\xi\in\g$.
The adjoint operator is denoted by $\Ad_g:\g\rightarrow\g$, and the ad operator is denoted by $\ad_\xi:\g\rightarrow\g$. See, for example~\cite{MarRat99} for detailed preliminaries. 

\subsection{Continuous-Time Extended Lagrangian Mechanics}

Consider a non-autonomous (left-trivialized) Lagrangian $L(t,g,\xi):\Re\times\G\times\g\rightarrow \Re$ on the \textit{extended state space}.
The corresponding \textit{extended path space} is composed of the curves $(c_t(a),c_g(a))$ on $\Re\times \G$ parameterized by $a>0$.
To ensure that the reparameterized time increases monotonically, we require $c'_t(a) > 0$. 
For a given time interval $[t_0,t_f]$, the corresponding interval $[a_0,a_f]$ for $a$ is chosen such that $t_0=c_t(a_0)$ and $t_f=c_t(a_f)$.
For any path $(c_t(a),c_g(a))$ over $[a_0,a_f]$ in the extended space, the \textit{associated curve} is 
\begin{align}
    g(t) = c_g(c_t^{-1}(t)),\label{eqn:ac}
\end{align}
on $\G$ over the time interval $[t_0,t_f]$.
For a given extended path, define the \textit{extended action integral} as
\begin{align}
    \mathfrak{G}(c_t,c_g) = \int_{t_0}^{t_f} L(t,g,\xi)\bigg|_{g(t) = c_g(c_t^{-1}(t))} dt,\label{eqn:AI}
\end{align}
where the Lagrangian is evaluated on the associated curve \eqref{eqn:ac}, and $\xi$ satisfies the kinematics equation  \eqref{eqn:g_dot}.

Taking the variation of $\mathfrak{G}$ with respect to the extended path, we obtain the Euler--Lagrange equation according to the variational principle in the extended phase space. 
As discussed in~\cite[Sec. 4.2.2]{MarWesAN01}, the resulting Euler--Lagrange equations depends only on the associated curve \eqref{eqn:ac}, not on the extended path $(c_t,c_g)$ itself, and the variational principle does not dictate how the curve should be reparameterized. 

Further, the resulting Euler--Lagrange equation share the exactly same form as (unextended) Lagrangian mechanics for the associated curve. 
As such, the Euler--Lagrange equation for non-autonomous Lagrangian $L(t,g,\xi):\Re\times\G\times\g\rightarrow \Re$ can be written as
\begin{align}
    \frac{d}{dt}\!\parenth{\deriv{L}{\xi}} - \ad^*_\xi \deriv{L}{\xi} - \T^*_e \L_g (\D_g L) = 0, \label{eqn:EL}
\end{align}
where $\D_g$ stands for the differential with respect to $g$ (see~\cite[Sec. 8.6.3]{LeeLeo17} for derivation of the above equation for autonomous Lagrangians).

Introducing the Legendre transform $\mu = \deriv{L}{\xi} \in\g^*$, and assuming that it is invertible, the Euler--Lagrange equation can be rewritten as
\begin{align}
    \dot \mu - \ad^*_{\xi} \mu - \T^*_e \L_g (\D_g L) = 0. \label{eqn:HE}
\end{align}

\subsection{Extended Lie Group Variational Integrator}

Variational integrators are geometric numerical integration schemes that can be viewed as discrete-time mechanics derived from a discretization of the variational principle for Lagrangian mechanics~\cite{MarWesAN01}.
The discrete-time flows of variational integrators are symplectic and they exhibit a discrete analogue of Noether's theorem.
This provides long-term structural stability in the resulting numerical simulations. 
For Lagrangian mechanics evolving on a Lie group, the corresponding Lie group variational integrators were developed in~\cite{LeeLeoCMAME07}.

Here, we develop extended Lie group variational integrators by discretizing the extended variational principle presented above, following the general framework of~\cite{MarWesAN01}.
The \textit{extended discrete path space} is composed of the sequence $\{(t_k, g_k)\}_{k=0}^N$ on $\Re\times\G$, satisfying $t_{k+1}>t_k$.
Next, the discrete kinematics equation is chosen to be
\begin{align}
    g_{k+1} = g_k f_k, \label{eqn:gkp}
\end{align}
for $f_k \in\G$ representing the relative update over a single timestep. 
The discrete Lagrangian $L_d(t_k, t_{k+1}, g_k, f_k): \Re\times\Re\times\G\times\G\rightarrow \Re$ is chosen such that the following \textit{extended discrete action sum}
\begin{align}
    \mathfrak{G}_d(\{(t_k, g_k)\}_{k=0}^N) = \sum_{k=0}^{N-1} L_d(t_k, t_{k+1}, g_k, f_k), \label{eqn:Gd}
\end{align}
approximates \eqref{eqn:AI}. 

\begin{prop}
    The discrete path $\{(g_k,f_k)\}_{k=0}^{N-1}$ that extremizes the discrete action sum \eqref{eqn:Gd} subject to fixed endpoints satisfies the following discrete Euler--Lagrange equation,
    \begin{gather}
        \T^*_e\L_{g_k}(\D_{g_k} L_{d_k})- \Ad^*_{f_k^{-1}} (\T^*_e\L_{f_k}(\D_{f_k} L_{d_k}))\nonumber \\
        + \T^*_e\L_{f_{k-1}}(\D_{f_{k-1}} L_{d_{k-1}}) =0,\label{eqn:DEL}\\
        \D_{t_k} L_{d_{k-1}} + \D_{t_k} L_{d_k} = 0, \label{eqn:DELt}
    \end{gather}
    which together with the discrete kinematic equation \eqref{eqn:gkp} defines an extended Lie group variational integrator.
\end{prop}
\begin{proof}
    From \eqref{eqn:gkp}, $\delta f_k = - g_k^{-1}( \delta g_k ) g_k^{-1} g_{k+1} + g_k^{-1}\delta g_{k+1}$.
    Since $\delta g_k$ can be written as $\delta g_k = g_k \eta_k $ for $\eta_k\in \g$, 
    \begin{align}
        f_k^{-1}\delta f_k = -\Ad_{f_k^{-1}} \eta_k + \eta_{k+1}.\label{eqn:del_fk}
    \end{align}

    Take the variation of \eqref{eqn:Gd} and substitute \eqref{eqn:del_fk} to obtain
    \begin{align*}
        \delta \mathfrak{G}_d  = \sum_{k=0}^{N-1}
        & \T^*_e\L_{g_k}(\D_{g_k} L_{d_k}) \cdot \eta_k \\
        & + \T^*_e\L_{f_k}(\D_{f_k} L_{d_k}) \cdot (-\Ad_{f_k^{-1}} \eta_k + \eta_{k+1}) \\
        & + \D_{t_k} L_{d_k}\cdot \delta t_k + \D_{t_{k+1}} \D_{d_k}\cdot \delta t_{k+1}.
    \end{align*}
    Since the endpoints are fixed, we have $\eta_0=0$ and $\delta t_0 = 0$.
    Therefore in the above expression, the range of summation for the terms paired with $\eta_k$ and $\delta t_k$ can be reduced to $1\leq k\leq N-1$. 
    Also, using $\eta_N=0$ and $\delta t_N=0$, for the other terms paired with $\eta_{k+1}$ and $\delta t_{k+1}$, the terms can be reindexed by reducing the subscripts by one and summed over the same range.
%    Therefore, 
    %\begin{align*}
        %\delta \mathfrak{G}_d  = \sum_{k=1}^{N-1}
        %& \big\{ \T^*_e\L_{g_k}(\D_{g_k} L_{d_k})- \Ad^*_{f_k^{-1}} (\T^*_e\L_{f_k}(\D_{f_k} L_{d_k})) \\
        %& + (\T^*_e\L_{f_{k-1}}(\D_{f_{k-1}} L_{d_{k-1}})) \big\} \cdot \eta_{k} \\
        %& + \{ \D_{t_k} L_{d_k} + \D_{t_{k}} \D_{d_{k-1}} \} \delta t_k.
%    \end{align*}
    According to the variational principle, $\delta\mathfrak{G}_d = 0$ for any $\eta_k$ and $\delta t_k$, which yields \eqref{eqn:DEL} and \eqref{eqn:DELt}.
\end{proof}
The most notable difference compared to the continuous-time counterpart is that in addition to the discrete Euler--Lagrange equation \eqref{eqn:DEL}, we have the additional equation \eqref{eqn:DELt} for the evolution of the discrete time. 
This is because the discrete action sum $\mathfrak{G}_d$ depends on the complete extended path $\{(t_k,g_k)\}_{k=1}^N$.
Whereas the continuous-time action $\mathfrak{G}$ is only a function of the associated curve \eqref{eqn:ac}.

The discrete Euler--Lagrange equation for the discrete time~\eqref{eqn:DELt} is associated with the energy.
Define the discrete energy to be
\begin{align}
    E^+_k &= - \D_{t_{k+1}} L_{d_k},\label{eqn:Ep}\\
    E^-_k &= \D_{t_{k}} L_{d_k}.\label{eqn:Em}
\end{align}
Then, \eqref{eqn:DELt} can be rewritten as
\begin{align}
    E^+_{k-1} = E^-_k,\label{eqn:DELE}
\end{align}
which reflects the evolution of the discrete energy. 
When the discrete Lagrangian is autonomous, \eqref{eqn:DELE} implies the conservation of discrete energy, thereby yielding a symplectic-energy-momentum integrator~\cite{KaMaOr1999}.

%This connection should come as no surprise, as the Hamiltonian is the conjugate momentum associated with time.
%In case the discrete Lagrangian is time-invariant so that $L_d(t_k,t_{k+1},g_k,f_k) =  L_d(t_k+s,t_{k+1}+s,g_k,f_k)$ for any $s$, 
%we have $\D_{t_k}L_{d_k} + \D_{t_{k+1}}L_{d_k}=0$, or equivalently $E^-_k = E^+_k$. 
%Combined with \eqref{eqn:DELE}, this yields $E_{k-1}^+ = E_k^+$ and $E_{k-1}^-=E_k^-$, which shows the conservation of discrete energy for autonomous Lagrangian systems. This yields an example of a  symplectic-energy-momentum integrator~\cite{KaMaOr1999}, and it does not violate the theorem of Ge and Marsden~\cite{GeMa1988} as it is not a fixed timestep method. Rather, \eqref{eqn:DELE} prescribes the timestep to be taken.

To implement \eqref{eqn:DEL} and \eqref{eqn:DELt} as a numerical integrator, it is more convenient to introduce the \textit{extended discrete Legendre transforms}, $\mathbb{F}^\pm L_{d_k}: \Re\times\Re \times \G \times \G \rightarrow \Re\times \Re\times\G\times\g^*$ as
\begin{align}
    \mathbb{F}^+ L_{d_k} (t_k,t_{k+1}, g_k,f_k) & = (t_{k+1}, E_{k+1}, g_{k+1}, \mu_{k+1}),\\
    \mathbb{F}^- L_{d_k} (t_k,t_{k+1}, g_k,f_k) & = (t_k, E_k, g_{k}, \mu_{k}).
\end{align}
where
\begin{align}
    \mu_k & = -\T^*_e\L_{g_k}(\D_{g_k} L_{d_k})+ \Ad^*_{f_k^{-1}} (\T^*_e\L_{f_k}(\D_{f_k} L_{d_k})),\label{eqn:muk}\\
    \mu_{k+1} & = \T^*_e\L_{f_k} (\D_{f_k} L_{d_k}),\label{eqn:mukp}
\end{align}
and $E_{k+1}$ and $E_k$ are given by \eqref{eqn:Ep} and \eqref{eqn:Em}, respectively. 

The resulting discrete flow map is defined by $\mathbb{F}^+L_{d_k} \circ (\mathbb{F}L_{d_k})^{-1}$. 
More specifically, for given $(t_k, E_k, g_k, \mu_k)$, \eqref{eqn:Em} and \eqref{eqn:muk} are solved together for $t_{k+1},f_k$ with the constraint $t_{k+1}>t_k$.
Then, $(E_{k+1}, g_{k+1},\mu_{k+1})$ are computed by \eqref{eqn:Ep}, \eqref{eqn:gkp}, and \eqref{eqn:mukp}, respectively.
This yields the discrete flow map $(t_k, E_k, g_k, \mu_k)\rightarrow(t_{k+1}, E_{k+1}, g_{k+1}, \mu_{k+1})$ consistent with \eqref{eqn:DEL} and \eqref{eqn:DELt}.
While the flow map is expressed in terms of $E$ for convenience, the initial value of $E_0$ is often selected by choosing the initial timestep $h_0$ and calculating the corresponding value of $E_0$ through \eqref{eqn:Em}.
This inherits the desirable properties of variational integrators, and the group structure is also preserved through \eqref{eqn:gkp}.
%, in combination with an ansatz $f_k=\exp(h\xi_k)$, where $\xi_k\in\g$ in order to ensure that $f_k\in\G$.

\section{Bregman Lagrangian Systems on $\G$}\label{sec:Breg}

\newcommand{\obj}{\mathsf{f}}

Let $\obj:\G\rightarrow\Re$ be a real-valued smooth function on $\G$.
We focus on the optimization problem:
\begin{align}
    \min_{g\in\G} \obj (g).
\end{align}
A variational accelerated optimization scheme for the above problem was developed in~\cite{tao2020variational}, where the Nesterov accelerated gradient (NAG) descent on a finite-dimensional vector space was intrinsically generalized to a Lie group. 
In this section, we introduce an intrinsic formulation of Bregman Lagrangian dynamics~\cite{wibisono2016variational}, which encompasses a larger class of accelerated optimization scheme, including NAG.
More importantly, the continuous dynamics guarantees polynomial convergence rates up to an arbitrary order.
%(note that this is not preserved under discretization; however, in practice, the numerical performance can be tuned; see Sec.\ref{sec:experimentA}, \cite{su2016differential} and also \cite{attouch2018fast}).

\subsection{Continuous-Time Bregman Dynamics}

The Bregman Lagrangian $L(t,g,\xi):\Re\times\G\times\g\rightarrow$ is %given by
\begin{align}
    L(t,g,\xi) = \frac{t^{\lambda p+1}}{2p} \|\xi\|^2 - C p t^{(\lambda+1)p-1} \obj (g),\label{eqn:BL}
\end{align}
where $\|\xi\|^2 = \met{\xi,\xi}=\pair{\mathbf{J}(\xi),\xi}$, for $p, C>0$, and $\lambda\geq 1$.
When $\G=\Re^n$ and $\lambda=1$, this recovers the Bregman Lagrangian for vector spaces~\cite{wibisono2016variational}, and it yields the continuous-time limit of Nesterov's accelerated gradient descent for $p=2$~\cite{nesterov2005smooth}.
Also, in case $p = 3$, it corresponds to the continuous-time limit of Nesterov’s accelerated cubic-regularized Newton’s method~\cite{nesterov2008accelerating}.
When $\G$ is considered as a Riemannian manifold, this corresponds to the $p$-Bregman Lagrangian in~\cite{duruisseaux2021variational}.
The additional term $\lambda$ accounts for the sectional curvature and diameter of the manifold~\cite{alimisis2020continuous}. 

The left-trivialized derivative of the objective function is
\begin{align}
    \nabla_\L \obj(g) = \T_e^*\L_g (\D_g \obj(g)).\label{eqn:grad}
\end{align}
Applying \eqref{eqn:EL} to \eqref{eqn:BL}, the corresponding Euler--Lagrange equations are given below.
\begin{prop}\label{prop:EL_Breg}
    The Euler--Lagrange equations corresponding to the Bregman Lagrangian \eqref{eqn:BL} are% given by
    \begin{align}
        \frac{d \mathbf{J}(\xi)}{dt} + \frac{\lambda p+1}{t}\mathbf{J}(\xi) - \ad^*_\xi \mathbf{J}(\xi)
        + Cp^2 t^{p-2} \nabla_\L \obj (g) =0, \label{eqn:EL_Breg}
    \end{align}
    and \eqref{eqn:g_dot}. 
    Further, the corresponding continuous flow locally converges to the minimizer $g^*$ of $\obj$ with the rate given by
    \begin{align}
        \obj(g(t)) - \obj(g^*) \in \mathcal{O}(t^{-p}),
    \end{align}
    when $\obj$ is geodesically convex. 
\end{prop}
\begin{proof}
    We have
    \begin{align*}
        \deriv{L}{\xi} & = \frac{t^{\lambda p+1}}{p}\mathbf{J}(\xi)
    \end{align*}
    Substituting this into \eqref{eqn:EL} and using \eqref{eqn:grad}, 
    \begin{gather*}
        \frac{t^{\lambda p+1}}{p}\frac{d \mathbf{J}(\xi)}{dt} + \frac{(\lambda p+1)t^{\lambda p}}{p}\mathbf{J}(\xi) -\frac{t^{\lambda p+1}}{p} \ad^*_\xi \mathbf{J}(\xi) \\
        + Cpt^{(\lambda +1)p-1} \nabla_\L \obj(g) =0.
    \end{gather*}
    Dividing both sides by $\frac{t^{\lambda p+1}}{p}$ yields \eqref{eqn:EL_Breg}.
    The convergence property is established by~\cite[Theorem 3.2]{duruisseaux2021variational}.
\end{proof}

Therefore, the optimization problem on $\G$ can be addressed by numerically integrating \eqref{eqn:EL_Breg} from an initial guess. 
However, it has been observed that a na\"\i ve discretization is not able to match the polynomial convergence rate established in~\cite{wibisono2016variational}.
Further, we need a guarantee that the discrete trajectory evolves on the Lie group.

These two challenges can be addressed by applying a Lie group variational integrator, 
as their structure-preserving properties provides long-term numerical stability, and preservation of the group structure.
In the subsequent section, we derive Lie group variational integrators for the Bregman Lagrangian system. 

\subsection{Lie Group Variational Integrator for Bregman Dynamics}

Let $h_k = t_{k+1}- t_k$ and $t_{k,k+1} = t_k + h_k/2$.  
We consider the following form of the discrete Lagrangian
\begin{align}
    L_d(t_k, t_{k+1}, g_k, f_k ) & = \frac{\phi(t_{k,k+1}) }{h_k} T_d(f_k) -\frac{h_k}{2} \theta(t_k) \obj (g_k)\nonumber \\
                                 & \quad -\frac{h_k}{2} \theta(t_{k+1}) \obj (g_kf_k), \label{eqn:Ld}
\end{align}
where $T_d(f_k):\G\rightarrow\Re$ is chosen such that it approximates $T(f_k) \approx h_k^2 \|\xi_k \|^2/2$, and $\phi,\theta:\Re\rightarrow\Re$ are 
\begin{align}
    \phi(t) & = \frac{t^{\lambda p +1}}{p},\\
    \theta(t) & = C p t^{(\lambda+1)p-1}.
\end{align}
The corresponding variational integrators are presented as follows.
\begin{prop}\label{prop:DEL_Breg}
    The discrete-time Euler--Lagrange equations, or the Lie group variational integrator for the discrete Lagrangian \eqref{eqn:Ld} corresponding to the Bregman Lagrangian~\eqref{eqn:BL} are given by
\begin{align}
    \mu_k & =  \frac{\phi_{k,k+1}}{h_k}\Ad^*_{f_k^{-1}} (\T^*_e\L_{f_k}(\D_{f_k} T_{d_k})) + \frac{h_k\theta_k}{2} \nabla_L \obj_k, \label{eqn:muk_Breg}\\
    \mu_{k+1} & = \Ad^*_{f_k}( \mu_k - \frac{h_k\theta_k}{2} \nabla_L\obj_k ) -\frac{h_k\theta_{k+1}}{2} \nabla_\L \obj_{k+1}, \label{eqn:mukp_Breg}\\
    E_{k} & = \frac{\phi'_{k,k+1}}{ 2 h_k} T_{d_k} -\frac{h_k\theta'_{k}}{2} \obj_{k} \nonumber\\\
          & \quad + \frac{\phi_{k,k+1}}{h_k^2}T_{d_k} + \frac{\theta_k}{2}\obj_k + \frac{\theta_{k+1}}{2}\obj_{k+1}, \label{eqn:Ek_Breg}\\
    E_{k+1} & = -\frac{\phi'_{k,k+1}}{ 2 h_k} T_{d_k} +\frac{h_k\theta'_{k+1}}{2} \obj_{k+1}\nonumber \\\
            & \quad + \frac{\phi_{k,k+1}}{h_k^2}T_{d_k} +\frac{\theta_k}{2}\obj_k + \frac{\theta_{k+1}}{2}\obj_{k+1}, \label{eqn:Ekp_Breg}
\end{align}
together with \eqref{eqn:gkp}.
%, where
%\begin{align*}
    %\phi'(t) & = \frac{(\lambda p +1) t^{\lambda p}}{p}, \\
    %\theta'(t) &  = Cp((\lambda+1)p-1) t^{(\lambda +1)p-2}.
%\end{align*}
\end{prop}
\begin{proof}
    These can be derived by substituting \eqref{eqn:Ld} into \eqref{eqn:muk}, \eqref{eqn:mukp}, \eqref{eqn:Em}, and \eqref{eqn:Ep}, respectively. 
\end{proof}
As discussed at the end of Section III, these provide symplectic and momentum-preserving discrete time flow maps.
Since these corresponds to a discretization of the Bregman Lagrangian system, they can be considered as a geometric numerical integrator for \eqref{eqn:EL_Breg}, or utilized as an optimization algorithm on $\G$.
If $T_d(f_k)=T_d(f_k^{-1})$, then the discrete Lagrangian is self-adjoint, and the above integrator is symmetric and therefore at least second-order accurate.

\section{Optimization on $\G$}

In this section, we present both of the continuous Bregman Lagrangian system and the Lie group variational integrator for several Lie groups.

\subsection{Euclidean Space $\Re^n$}

Suppose $\G=\Re^n$, with the additive group action, and the inner product is chosen to be $\pair{x,y}=x^Ty$ for any $x,y\in\Re^n$.
Let $\mathbf{J}(\dot x) = I_{n\times n} \dot x$, and $\lambda =1$. 

From \eqref{eqn:EL_Breg}, the continuous Euler--Lagrange equation is given by
\begin{align}
    \ddot x + \frac{p+1}{t} \dot x + Cp^2 t^{p-2} \nabla \obj(x) = 0,\label{eqn:EL_Re}
\end{align}
which recovers the differential equation derived in \cite{wibisono2016variational}.

Next, we develop variational integrators. 
The discrete kinematics equation \eqref{eqn:gkp} is rewritten as $x_{k+1} = x_k + \Delta x_k$ for $\Delta x_k\in\Re^n$.
The kinetic energy term in \eqref{eqn:Ld} is chosen as
\begin{align}
    T_d =\frac{1}{2}\|\Delta x_k\|^2.\label{eqn:Td_Re}
\end{align}
According to \Cref{prop:DEL_Breg}, we obtain the discrete Euler--Lagrange equations as follows.

\begin{prop}
    When $G=\Re^n$, the variational integrator for the discrete Bregman Lagrangian \eqref{eqn:Ld} is given by 
\begin{align}
    v_k & =  \frac{\phi_{k,k+1}}{h_k}\Delta x_k + \frac{h_k\theta_k}{2} \nabla \obj_k,\label{eqn:muk_Re} \\
    v_{k+1} & = v_k - \frac{h_k\theta_k}{2} \nabla\obj_k  -\frac{h_k\theta_{k+1}}{2} \nabla \obj_{k+1},  \label{eqn:mukp_Re}
%    E_{k} & = \frac{\phi'_{k,k+1}}{ 4 h_k} \|\Delta x_k\|^2 -\frac{h_k\theta'_{k}}{2} \obj_{k} \nonumber\\\
          %& \quad + \frac{\phi_{k,k+1}}{2 h_k^2}\|\Delta x_k\|^2 + \frac{\theta_k}{2}\obj_k + \frac{\theta_{k+1}}{2}\obj_{k+1},\label{eqn:Ek_Re}\\
    %E_{k+1} & = -\frac{\phi'_{k,k+1}}{ 4 h_k} \|\Delta x_k\|^2 +\frac{h_k\theta'_{k+1}}{2} \obj_{k+1}\nonumber \\\
%            & \quad + \frac{\phi_{k,k+1}}{2 h_k^2}\|\Delta x_k\|^2 +\frac{\theta_k}{2}\obj_k + \frac{\theta_{k+1}}{2}\obj_{k+1}.
\end{align}
and \eqref{eqn:Ek_Breg}, \eqref{eqn:Ekp_Breg} with \eqref{eqn:Td_Re}.
\end{prop}
These are implicit as \eqref{eqn:muk_Re} and \eqref{eqn:Ek_Breg} should be solved together for $\Delta x_k$ and $h_k$.
One straightforward approach is fixed-point iteration.
For a given $h_k$, \eqref{eqn:muk_Re} can be solved explicitly for $\Delta_k$, which yields $x_{k+1}$. 
Then, \eqref{eqn:Ek_Breg} can be solved for $h_k$.
These procedure are iterated until $h_k$ converges. 

\subsection{Three-Dimensional Special Orthogonal Group $\SO$}

Next, consider $\SO=\{R\in\Re^{3\times 3}\,|\, R^T R = I_{3\times 3},\, \mathrm{det}(R)]=1\}$.
Its Lie algebra is $\so=\{S\in\Re^{3\times 3}\,|\, S^T=-S\}$ with the matrix commutator as the Lie bracket. 
This is identified with $\Re^3$ through the \textit{hat} map $\hat\cdot :\Re^3\rightarrow\so$ defined such that $\hat x\in\so$ and $\hat x y = x\times y$ for any $x,y\in\Re^3$.
The inverse of the hat map is denoted by the \textit{vee} map $\vee: \so\rightarrow\Re^3$.
The inner product is given by
\begin{align*}
    \langle \hat \eta, \hat \xi \rangle_{\so} = \frac{1}{2}\tr{\hat \eta^T\hat \xi} = \eta^T\xi = \pair{\eta,\xi}_{\Re^3}.
\end{align*}
The metric is chosen as
\begin{align}
    \langle \mathbf{J}(\hat \eta), \hat \xi \rangle_{\so} = \tr{\hat \eta^T J_d \hat \xi} = \eta^TJ \xi = \pair{J\eta,\xi}_{\Re^3},\label{eqn:met_SO3}
\end{align}
where $J\in\Re^{3\times 3}$ is a symmetric, positive-definite matrix, and $J_d=\frac{1}{2}\tr{J}I_{3\times 3}-J\in\Re^{3\times 3}$.
Further, 
\begin{alignat*}{2}
    \ad_\eta\xi &= \eta\times \xi,& \quad \ad^*_\eta \xi &= \xi\times \eta,\\
    \Ad_F\eta &= F\eta,& \quad \Ad^*_F \eta &= F^T\eta.
\end{alignat*}

Consider
\begin{align*}
    L(t,R,\Omega) = \frac{t^{p+1}}{2p} \Omega\cdot J\Omega - Cpt^{2p-1} \obj(R).
\end{align*}
From \eqref{eqn:EL_Breg}, the Euler--Lagrange equations are given by
\begin{gather}
    J\dot\Omega + \frac{p+1}{t} J\Omega + \hat\Omega J\Omega + C p^2 t^{p-2} \nabla_\L \obj(R) = 0,\label{eqn:EL_SO3}\\
    \dot R = R\hat\Omega. \label{eqn:R_dot}
\end{gather}

Next, we derive variational integrators. 
The kinematics equation is written as
\begin{align}
    R_{k+1} = R_k F_k, \label{eqn:Rkp}
\end{align}
for $F_k\in\SO$.
Similar with~\cite{LeeLeoCMAME07}, the angular velocity is approximated with $\hat\Omega_k \approx \frac{1}{h_k} R_k^T (R_{k+1}-R_k) = \frac{1}{h_k} (F_k-I_{3\times 3})$.
Substituting this into \eqref{eqn:met_SO3},
\begin{align}
    T_d(F_k) & =\tr{(I_{3\times 3}-F_k)J_d}, \label{eqn:Td_SO3}
\end{align}
which satisfies $T_d(F_k)=T_d(F_k^T)$.

\begin{prop}\label{prop:DEL_Breg_SO3}
    When $\G=\SO$, the Lie group variational integrator for the discrete Bregman Lagrangian \eqref{eqn:Ld} with \eqref{eqn:Td_SO3} is given by
\begin{align}
    \mu_k & =\frac{\phi_{k,k+1}}{h_k } (F_k J_d - J_dF_k^T)^\vee + \frac{h_k\theta_k }{2} \nabla_\L \obj_k, \label{eqn:muk_SO3} \\
    \mu_{k+1}  & = F_k^T \mu_k -\frac{h_k\theta_k}{2}\nabla_\L \obj_k - \frac{h_k\theta_k}{2} \nabla_\L \obj_{k+1}, \label{eqn:mukp_SO3}
%    E_{k} & = \frac{\phi'_{k,k+1}}{ 2 h_k} \tr{(I-F_k)J_d} -\frac{h_k\theta'_{k}}{2} \obj_{k} \nonumber\\\
%          & \quad + \frac{\phi_{k,k+1}}{h_k^2} \tr{(I-F_k)J_d} + \frac{\theta_k}{2}\obj_k + \frac{\theta_{k+1}}{2}\obj_{k+1}, \label{eqn:Ek_SO3}\\
%    E_{k+1} & = -\frac{\phi'_{k,k+1}}{ 2 h_k} \tr{(I-F_k)J_d} +\frac{h_k\theta'_{k+1}}{2} \obj_{k+1}\nonumber \\\
%            & \quad + \frac{\phi_{k,k+1}}{h_k^2} \tr{(I-F_k)J_d} +\frac{\theta_k}{2}\obj_k + \frac{\theta_{k+1}}{2}\obj_{k+1}, \label{eqn:Ekp_SO3}
\end{align}
together with \eqref{eqn:Ekp_Breg}, \eqref{eqn:Rkp}, \eqref{eqn:Ek_Breg},  and \eqref{eqn:Td_SO3}.
\end{prop}
\begin{proof}
    Let $\delta F_k = F_k \hat\chi_k$.
    The derivative of \eqref{eqn:Td_SO3} is %given by
    \begin{align*}
        \D_{F_k} T_{d_k} \cdot \delta F_k = \tr{-F_k\hat\chi_k J_d} = (J_dF_k - F_k^T J_d)^\vee \cdot \chi,
    \end{align*}
    where the last equality is from the identity, $\mathrm{tr}[-\hat x A]= x\cdot (A-A^T)^\vee$ for any $x\in\Re^3$ and $A\in\Re^{3\times 3}$.
    Thus, $\T^*_I \L_{F_k} (\D_{F_k} T_{d_k}) = (J_dF_k - F_k^T J_d)^\vee $. 
    Substituting this into \eqref{eqn:muk_Breg} and \eqref{eqn:mukp_Breg} yields \eqref{eqn:muk_SO3} and \eqref{eqn:mukp_SO3}, respectively. 
%    The remaining \eqref{eqn:Ek_SO3} and \eqref{eqn:Ekp_SO3} are directly from \eqref{eqn:Ek_Breg} and \eqref{eqn:Ekp_Breg} with \eqref{eqn:Td_SO3}.
\end{proof}
To implement these, \eqref{eqn:mukp_SO3} and \eqref{eqn:Ek_Breg} should be solved together for $h_k$ and $F_k$.
For a given $h_k$, computational approaches to solve \eqref{eqn:muk_SO3} for $F_k$ are presented in~\cite[Sec 3.3.8]{Lee08}.
When $J=I_{3\times 3}$, or equivalently when $J_d = \frac{1}{2}I_{3\times 3}$, \eqref{eqn:muk_SO3} can be solved explicitly to obtain
\begin{align}
    F_k = \exp \left(\frac{\sin^{-1}\|a\|}{\|a\|}\hat a\right),\label{eqn:Fk_SO3}
\end{align}
where $a = \frac{h_k}{\phi_{k,k+1}} (\mu_k - \frac{h_k\theta_k}{2}\nabla_\L \obj_k)\in\Re^3$.
This can replace \eqref{eqn:muk_SO3}.

\subsection{Product of $\Re^n$ and $\SO$}

Suppose $\G=\SO\times \Re^n$.
As it is the direct product of $\SO$ and $\Re^n$, the variation of the action sum is decomposed into two parts of $\SO$ and $\Re^n$. 
Therefore, the continuous Euler--Lagrange equations on $\SO\times\Re^n$ are given by \eqref{eqn:EL_Re} and \eqref{eqn:EL_SO3}, after replacing $\nabla \obj(x)$ of \eqref{eqn:EL_Re} with $\nabla_x \obj(R,x)$, and replacing $\nabla_\L f(R)$ of \eqref{eqn:EL_SO3} with $\T^*_I\L_R(\D_R\obj(R,x))$.

Similarly, the corresponding Lie group variational integrators are also given by \eqref{eqn:muk_Re}, \eqref{eqn:mukp_Re}, \eqref{eqn:muk_SO3}, and \eqref{eqn:mukp_SO3}, in addition to the energy equations \eqref{eqn:Ek_Breg} and \eqref{eqn:Ekp_Breg} with
\begin{align*}
    T_{d_k}(F_k,\Delta x_k) = \frac{1}{2}\|\Delta x_k\|^2 + \tr{(I_{3\times 3}-F_k)J_d}.
\end{align*}
%For $n=3$, these can be utilized for optimization on $\SE$ as well, since $\SE$ and $\SO\times \Re^3$ are identical as manifolds, although Bregman Lagrangian systems can be directly developed on $\SE$, where the critical difference would be in the choice of metric.

\section{Numerical Examples}

\subsection{Optimization on $\SO$}

\label{sec:experimentA}

Consider the objective function given by
\begin{align}
    \obj (R) & = \frac{1}{2}\| A- R\|^2_{\mathcal{F}}  %= \frac{1}{2} \tr{(A-R)^T(A-R)}  \\
    =\frac{1}{2}(\|A\|^2_{\mathcal{F}} + 3) - \tr{A^T R},\label{eqn:obj_SO3}
\end{align}
where $\|\cdot\|_{\mathcal{F}}$ denotes the Frobenius norm, and $A\in\Re^{3\times 3}$.
Optimization of the above function appears in the least-squares estimation of attitude, referred to as Wahba's problem~\cite{WahSR65}. 
Let the singular value decomposition of $A=USV^T$ for a diagonal $S\in\Re^{3\times 3}$ and $U,V\in\mathsf{O}(3)$.
The optimal attitude is explicitly given by $R^* =  U \mathrm{diag}[1,1,\mathrm{det}(UV)] V^T$.
The left-trivialized gradient is $\nabla_\L \obj(R)  = (A^T R - R^T A)^\vee$.

\subsubsection{Order of Convergence}

\begin{figure}
    \centerline{
        \subfigure[convergence with respect to $t$]{\includegraphics[width=0.8\columnwidth]{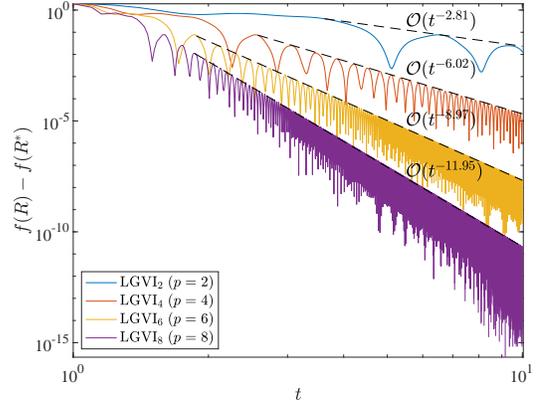}}
    }
    \centerline{
        \subfigure[convergence with respect to $k$]{\includegraphics[width=0.8\columnwidth]{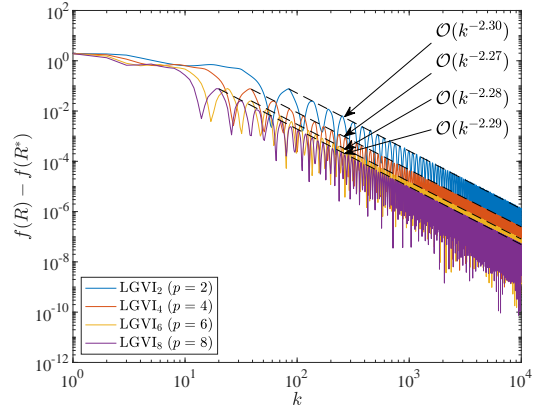}}
    }
    \caption{Convergence rate of LGVI in \Cref{prop:DEL_Breg_SO3} for varying $p$}\label{fig:conv}
\end{figure}

First, we check if the theoretical order of convergence guaranteed by \Cref{prop:EL_Breg} is achieved by the discrete Euler--Lagrange equations presented in \Cref{prop:DEL_Breg}.
The elements of the matrix $A$ in \eqref{eqn:obj_SO3} are randomly chosen from the uniform distribution on $[0,1]$.
The initial guess of $R_0$ is chosen such that the initial error is $0.9\pi$ in terms of the Euler-axis rotation. 
Lie group variational integrators (LGVI) in \Cref{prop:DEL_Breg_SO3} are simulated with fixed $J=I_{3\times 3}$, $C=1$, and $h_0 = 0.1$ for varying $p\in\{2,4,6,8\}$. 
Since $J=I_{3\times 3}$, \eqref{eqn:muk_SO3} is replaced by \eqref{eqn:Fk_SO3}.
The remaining implicit equation \eqref{eqn:Ek_Breg} is solved for $h_k$ via the Matlab equation solver, \texttt{lsqnonlin} with the tolerance of $10^{-4}$.
The initial guess for $h_k$ is provided by $h_{k-1}$. 

The resulting convergence rate represented by $\obj-\obj^*$ over $t_k$ is illustrated in \Cref{fig:conv}.(a), where the empirical convergence rate computed by manual fitting are also marked. 
It is shown that LGVI empirically achieved the order of convergence greater than the theoretical guarantee of $\mathcal{O}(t^{-p})$.
It has been reported that na\"\i ve discretizations of Bregman Lagrangian systems are not able to match the theoretical convergence rate, or it might cause numerical instability~\cite{wibisono2016variational,betancourt2018symplectic}.
These results suggest that LGVIs do not suffer from these discretization issues, and their performance are consistent with the continuous-time analysis. 

Next, given that the step size $h_k$ is adjusted adaptively according to \eqref{eqn:Ek_Breg} and \eqref{eqn:Ekp_Breg}, it is likely that numerical simulation with higher $p$ requires a smaller step size. 
In fact, the average step sizes are given by $6.15\times10^{-2}$, $6.50\times 10^{-3}$, $4.89\times10^{-4}$ and $1.21\times 10^{-5}$, respectively for $p\in\{2,4,6,8\}$. 
To examine the effects of the step size variations, the convergence with respect to the discrete time step is illustrated in \Cref{fig:conv}.(b).
It turns out that all of four cases of $p$ exhibit the similar order of long-term convergence, approximately $\mathcal{O}(k^{-2.3})$. This is not surprising, as Nesterov~\cite{nesterov2003introductory} showed that for every smooth first-order method, there exists a convex, $L$-smooth objective function, such that the rate of convergence is bounded from below by $\mathcal{O}(k^{-2})$, but it does not preclude the possibility of faster rates of convergence for strongly convex functions.

However, the case of higher $p$ benefits from faster initial convergence, and as a result, the terminal error for $p=4$ is more than 400 times smaller than that of $p=2$.

\subsubsection{Effects of Initial Step Size}

\begin{figure}
    \centerline{
        \subfigure[convergence with respect to $t$]{\includegraphics[width=0.8\columnwidth]{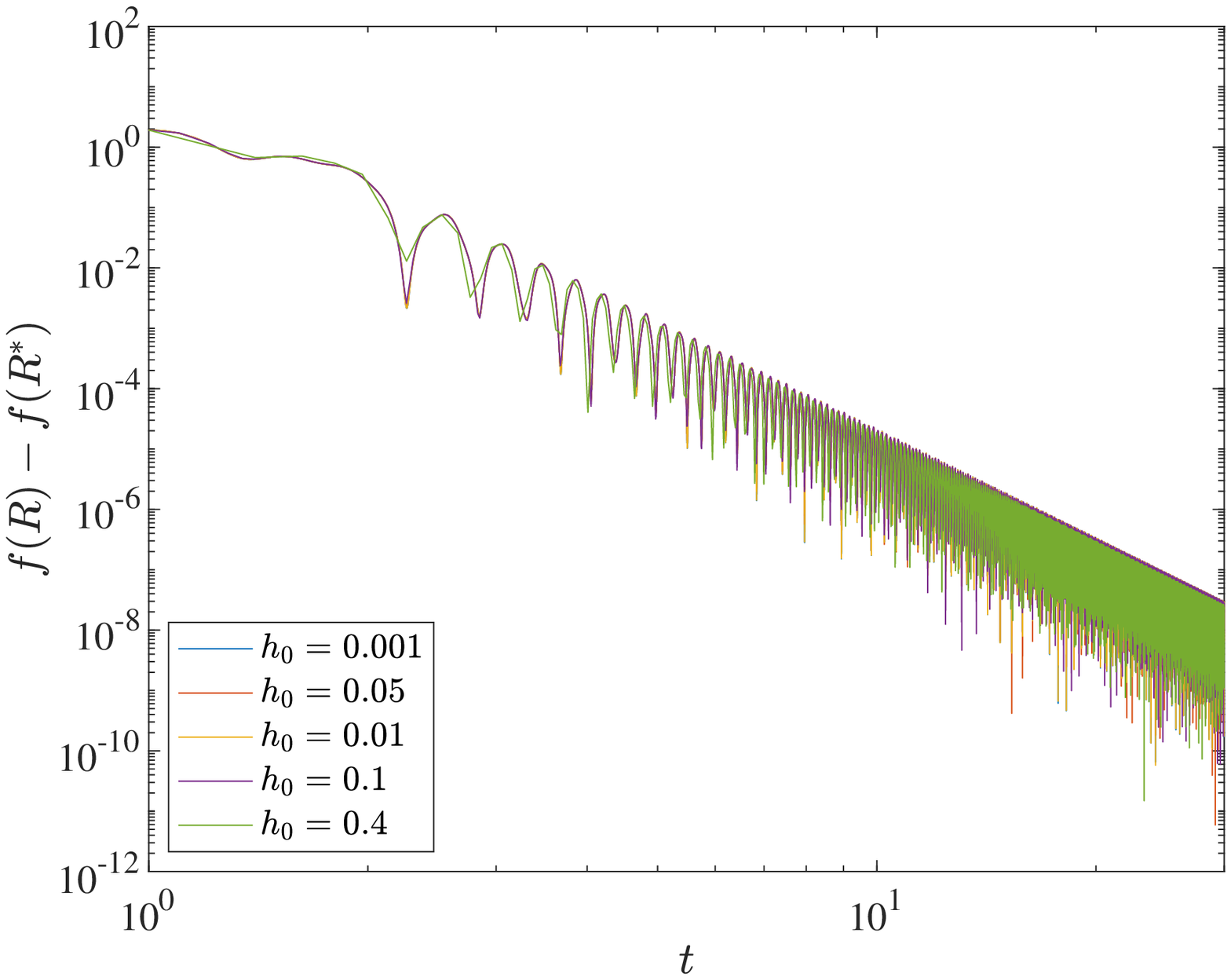}}
    }
    \centerline{
        \subfigure[evolution of step size $h_k$]{\includegraphics[width=0.8\columnwidth]{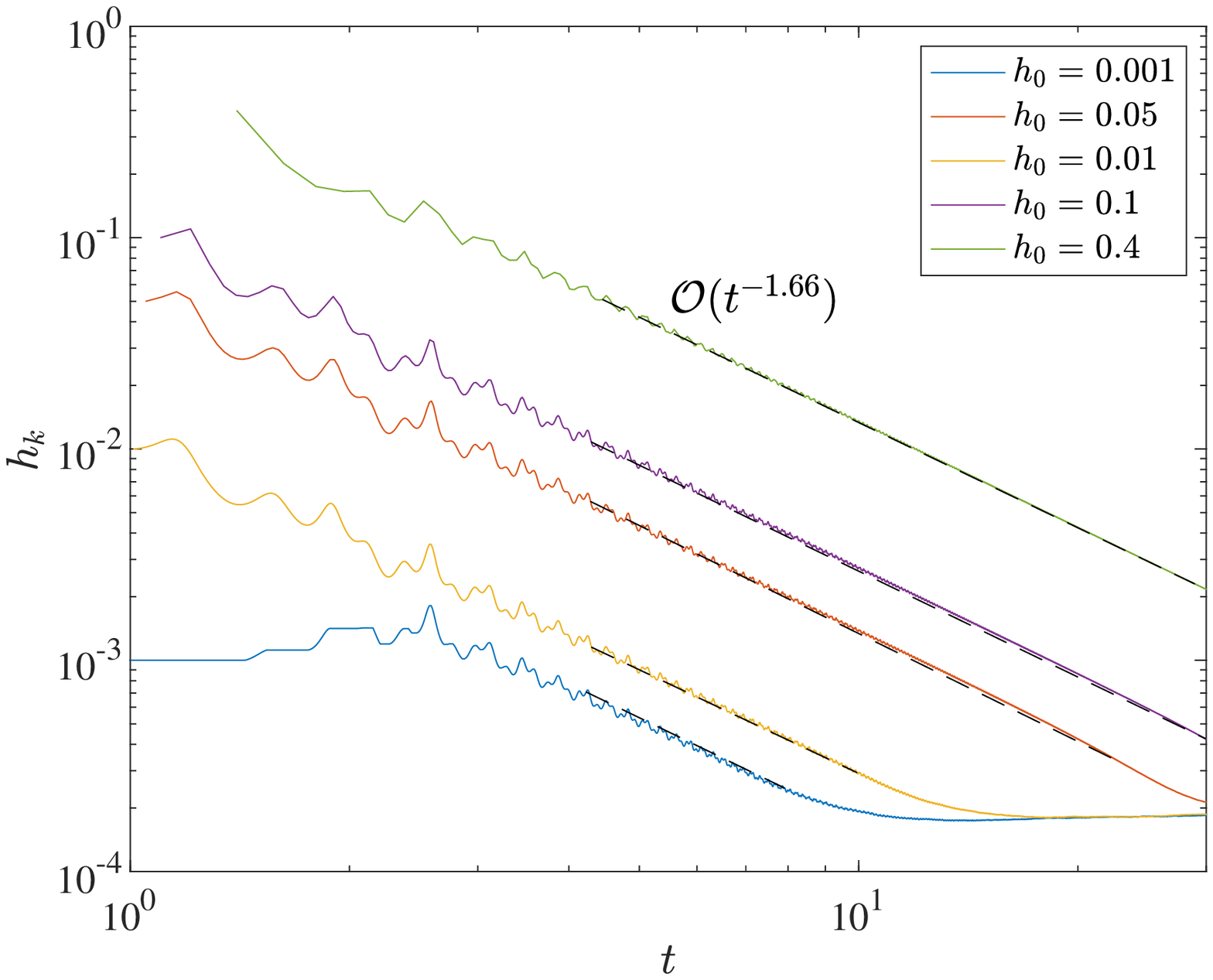}}
    }
    \caption{Convergence rate of LGVI in \Cref{prop:DEL_Breg_SO3} for varying $h_0$.} \label{fig:conv_h}
\end{figure}

As discussed at the end of \Cref{sec:Breg}, the extended LGVI requires choosing the initial step size $h_0$. 
Here, we study the effects of $h_0$ in the convergence.
More specifically, the order is fixed to $p=4$, and the initial step size is varied as $h_0\in\{0.001, 0.05, 0.01, 0.1, 0.4\}$.
The corresponding results are illustrated at \Cref{fig:conv_h}.
Interestingly, in \Cref{fig:conv_h}.(a), the convergence with respect to $t$ is not much affected by the initial step size $h_0$.
Next, \Cref{fig:conv_h}.(b) presents the time-evolution of the step size, and it is shown that the step size computed by \eqref{eqn:Ek_Breg} decreases at the approximate order of $\mathcal{O}(t^{-1.6})$  for all cases. 
This might have been caused by the fact that the forcing term in \eqref{eqn:EL_SO3} increases over time.
Another notable feature is that after a certain period, the step sizes tend to converge. 
More specifically, the step size initialized by $h_0=0.001$ converges to $1.8\times 10^{-4}$ when $t>10$, which is joined by the case of $h_0=0.005$ later. 
It is expected that the next case for $h_0=0.01$ would follow the similar trend if the simulation time is increased.
This implies a certain stability property of the extended LGVI in the step size.
Furthermore, observe that for the wide range of variations of step sizes presented in \Cref{fig:conv_h}.(b), the convergence in \Cref{fig:conv_h}.(a) is fairly consistent, which suggests that the LGVI is robust to the choice of the step size. 

\subsubsection{Comparison with Other Discretizations of Bregman Euler--Lagrange Equation}

\begin{figure}
    \centerline{
        \subfigure[convergence with respect to $t$]{\includegraphics[width=0.8\columnwidth]{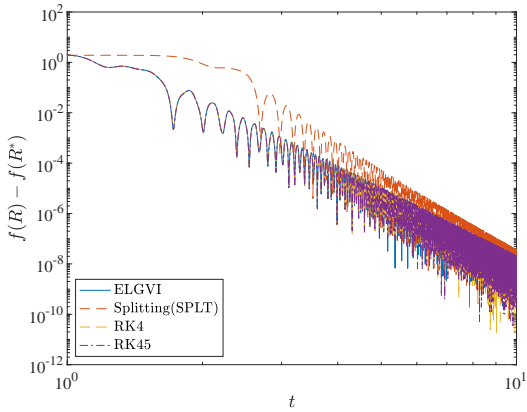}}
    }
    \centerline{
        \subfigure[orthogonality error of $R_k$]{\includegraphics[width=0.8\columnwidth]{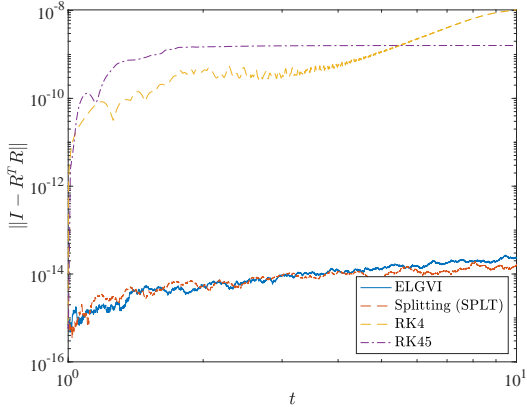}}
    }
    \caption{Comparison with other discretization schemes for Bregman Euler--Lagrange equation}\label{fig:comp_disc}
\end{figure}

Next, we compare LGVI with other discretization schemes applied to \eqref{eqn:EL_SO3} and \eqref{eqn:R_dot}.
Three methods are considered, namely the splitting approach introduced in \cite{tao2020variational} applied to the proposed continuous dynamics (abbreviated as SPLT), a 4-th order fixed-step Runge--Kutta method (RK4), and a variable stepsize Runge--Kutta method (RK45) implemented by the Matlab \texttt{ode45} function with the tolerance of $10^{-8}$. 
More precisely, the evolution of SPLT over step size $h$ is written as $\phi_{h/2} \circ \psi_h \circ \phi_{h/2}$, where $\phi_t$ is the exact flow map of \eqref{eqn:R_dot} with fixed $\Omega$, and $\psi_t$ is the exact $t$-time flow map of \eqref{eqn:EL_SO3} with fixed $R$ and $J=I_{3\times 3}$.

The goal of this comparison is not to claim that a certain method is superior to the other methods.
Rather, it is to identify the numerical properties of LGVI compared with others. 
Having stated that, LGVI is implicit, and \eqref{eqn:Ek_Breg} is solved by a general purpose nonlinear solver, instead of a numerical solver tailored for \eqref{eqn:Ek_Breg}.
As a consequence, LGVI is substantially slower than the three explicit methods, to the extent that the comparison is not meaningful. 

Instead, for a more interesting comparison, we exploit the property of LGVI providing consistent results for a wide range of step sizes, and we only utilize \eqref{eqn:muk_SO3} and \eqref{eqn:mukp_SO3} with a fixed prescribed step size. 
The resulting scheme, denoted by ELGVI, is explicit as shown in \eqref{eqn:Fk_SO3}.
Overall ELGVI is quite comparable with SPLT, but it benefits from a bit faster initial convergence, especially when $p$ is larger and $h$ is smaller. 
One particular case for $p=6$ and $h=0.001$ is illustrated in \Cref{fig:comp_disc}.(a).
With regard to RK4 and RK45, their convergence is almost identical to ELGVI, but as presented in \Cref{fig:comp_disc}.(b), those methods do not preserve the orthogonality of the rotation matrix, which is problematic.  
Whereas, both of LGVI and SPLT conserve the structure of rotation matrices. 
Next, the computation time with Intel Core i7 3.2GHz, averaged for 10 executions, are 0.0727, 0.0258, 0.3847, and 1.1476 seconds for ELGVI, SPLT, RK4, and RK45, respectively. 
It is expected that RK4 requires more computation time as the gradient should be evaluated four times per a step, and it seems that the time-adaptive RK45 algorithm requires more frequent evaluations of the gradient. 

\subsubsection{Comparison with Other Optimization Schemes on Lie Groups}

\begin{figure}
    \centerline{
        \includegraphics[width=0.8\columnwidth]{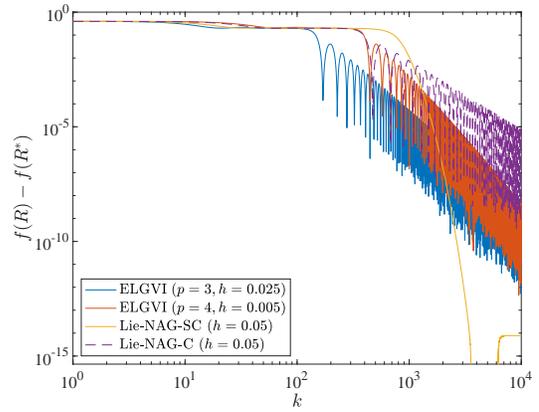}
    }
    \caption{Comparison with other accelerated optimization schemes on Lie groups}\label{fig:comp_accel}
\end{figure}

Finally, we compare ELGVI with other optimization schemes on Lie groups. 
In particular, we consider variationally accelerated Lie-group methods based on the NAG variational principle and operating splitting \cite{tao2020variational},  referred to as Lie-NAG-SC and Lie-NAG-C, which are conformally symplectic and group-structure preserving. Note that Lie-NAG-C corresponds to SPLT with $p=2$. %(the choice of NAG).

Four cases are considered as marked in \Cref{fig:comp_accel} for varying $p$ and $h$.
Compared with Lie-NAG-C, ELGVI exhibits faster convergence at a higher order. 
This does not contradict Nesterov's oracle lower bound: the continuous Bregman dynamics with $p>2$ should be discretized by smaller steps as $t$ increases, and therefore, the asymptotic order of convergence is still $\mathcal{O}(1/k^2)$ as illustrated above.
However, since ELGVI uses a fixed stepsize, the initial error can decay faster than inverse quadratic, and depending on the level of accuracy required, we can take the advantage of it by employing early stopping.
On the other hand, Lie-NAG-SC demonstrates exponential convergence asymptotically when applied to strongly convex functions. 
Overall, if  moderate stopping criteria are employed, ELGVI may be preferred, as they exhibit the fastest initial decay of the cost function.
%ELGVI may benefit from early stopping if moderate accuracy is acceptable, as they enjoy the best initial convergence. 

\subsection{Optimization on $\SO\times \Re^3$}

Next, we present an optimization problem on $\SO\times\Re^3$ to estimate the position and the attitude of a camera using the KITTI vision benchmark dataset~\cite{Geiger2013IJRR}.
This is to verify the performance of ELGVI for a non-convex function in a higher-dimensional Lie group, with more relevance to engineering practice. 
More specifically, we consider $N=516$ distinct features on a single image frame, where their 2D pixel coordinates in the image plane, and the actual 3D location in the world coordinates are given by $p^i\in\Re^3$ and $P^i\in\Re^4$, respectively as homogeneous coordinates. 
Assuming that the camera calibration matrix $K\in\Re^{3\times 3}$ is also known, we wish to estimate the pose $(R,x)\in\SO\times \Re^3$ of the camera.  

This is formulated as an optimization problem to minimize the reprojection error, which is the discrepancy between the actual pixel location of the features and the features projected to the image plane by the current estimate of $(R,x)$~\cite{ma2012invitation}.
For example, let $\tilde p^i\in\Re^3$ be the homogeneous coordinates for the feature corresponding to $P^i$ projected to the image plane by $(R,x)$. 
From the perspective camera model, 
\begin{align*}
    \lambda \tilde p^i = K[R, x]P^i,
\end{align*}
for $\lambda >0$.
The corresponding reprojected pixel is determined by the dehomogenization  of $\tilde p^i$, namely $H^{-1}(\tilde p^i)\in\Re^2$ corresponding to the first two elements of $\tilde p^i$ divided by the last element. 
The objective function is the sum of the reprojection error given by
\begin{align}
    \obj (R,x) =  \sum_{i=1}^N \| H^{-1}(p^i) - H^{-1}(\tilde p^i)\|^2.\label{eqn:obj_SE3}
\end{align}
\Cref{fig:comp_6} presents the optimization results by ELGVI, which are comparable to the benchmark examples presented for $\SO$. 
However, the terminal phase is relatively noisy, partially because the gradients of \eqref{eqn:obj_SE3} are evaluated numerically with a finite-difference rule. 
\Cref{fig:KITTI} illustrates the reprojected features before and after the optimization. 

\begin{figure}
    \centerline{
        \includegraphics[width=0.8\columnwidth]{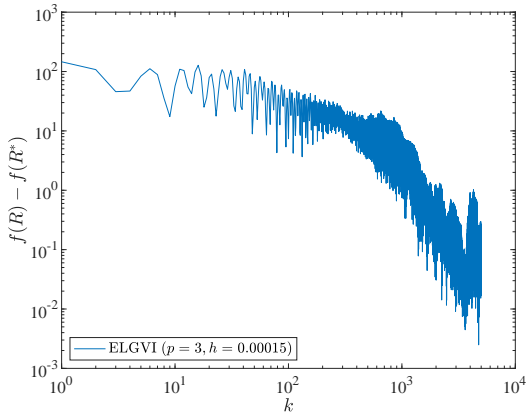}
    }
    \caption{Optimization on $\SO\times\Re^3$: convergence with respect to $k$}\label{fig:comp_6}
\end{figure}

\begin{figure}
    \centerline{
        \subfigure[Initial guess $(R_0,x_0)$]{\includegraphics[width=1.0\columnwidth]{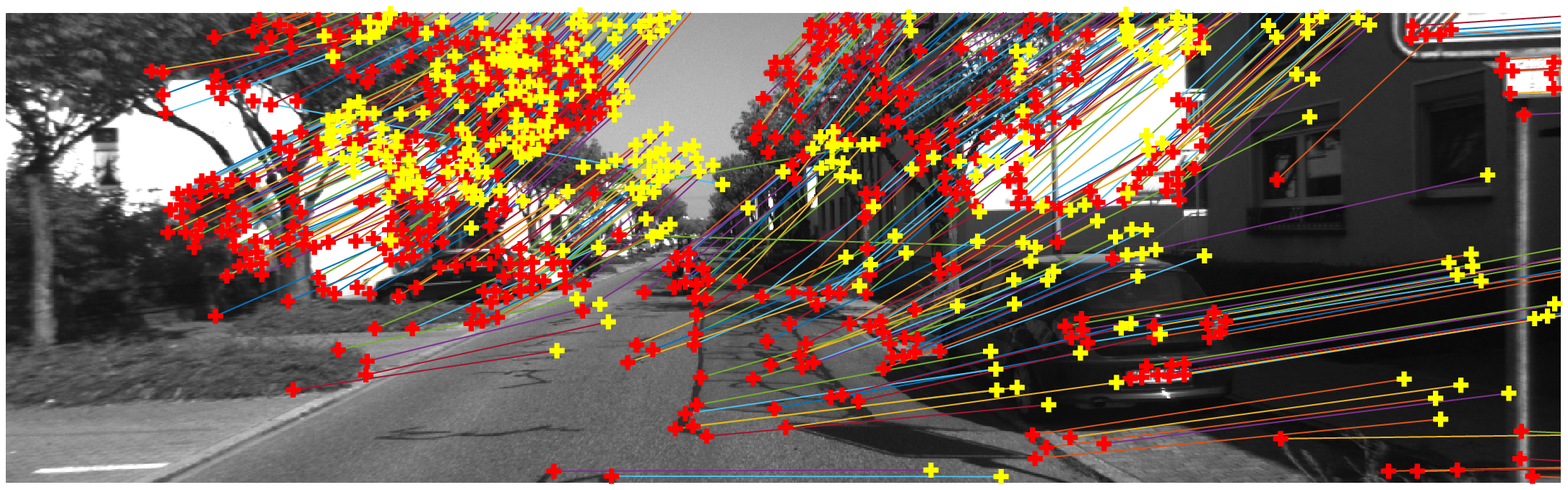}}
    }
    \centerline{
        \subfigure[Optimized $(R^*,x^*)$]{\includegraphics[width=1.0\columnwidth]{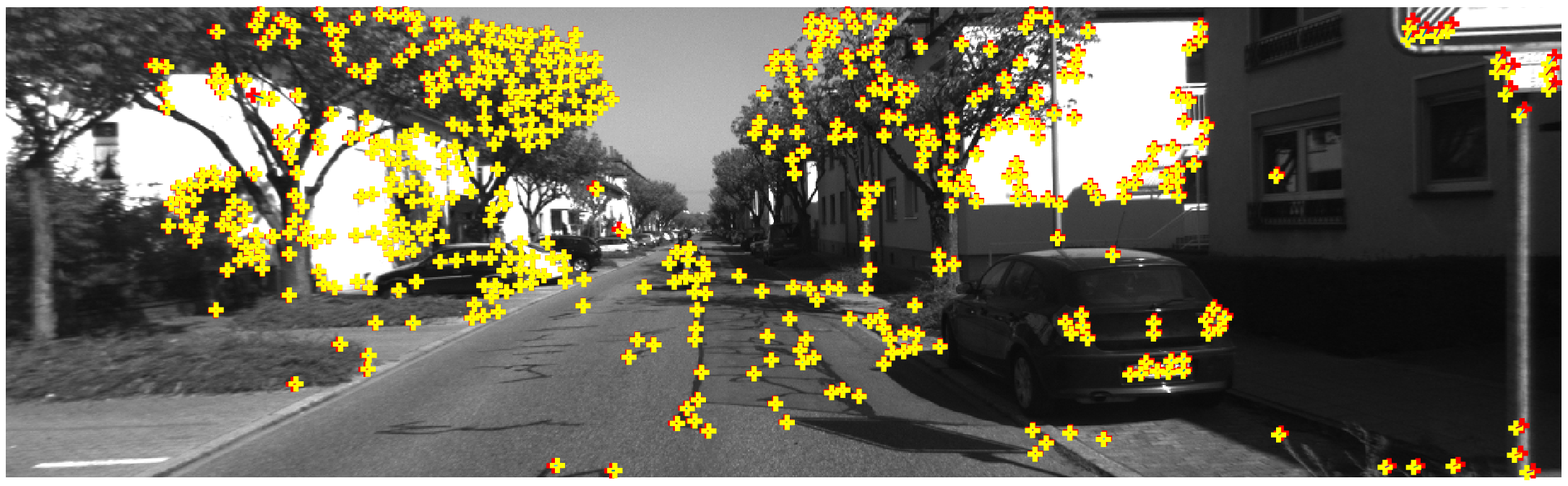}}
    }
    \caption{Reprojection error: the red $+$ markers denote the key points detected, and the yellow $+$ markers represent the key points projected by the estimated pose. The paired features are connected by solid lines. }\label{fig:KITTI}
\end{figure}

\section{Conclusions}

In this paper, we proposed a Lie group variational integrator for the Bregman Lagrangian dynamics on Lie groups, to construct an accelerated optimization scheme.
The variable stepsize prescribed by the extended variational principle exhibits an interesting convergence property, and the variational discretization is robust to the initial stepsize.
It would be interesting to explore the role of variable time-stepping in geometric discretizations of the Bregman dynamics especially compared with Hamiltonian variational integrators.

\bibliography{CDC21}
\bibliographystyle{IEEEtran}

\end{document}